\documentclass[a4paper,11pt]{article}
\usepackage{amsmath,amssymb,amsfonts,amsthm}
\usepackage{bm}
\usepackage[title]{appendix}
\usepackage{caption}
\usepackage{graphicx,subfig}
\usepackage{color,eucal,enumerate,mathrsfs}
\usepackage[normalem]{ulem}
\usepackage[pagewise]{lineno}
\usepackage[colorlinks=true,linkcolor=blue,anchorcolor=blue,citecolor=blue,urlcolor=blue]{hyperref}

\numberwithin{equation}{section}
\theoremstyle{plain}
\newtheorem{theorem}{Theorem}[section]
\newtheorem{corollary}[theorem]{Corollary}
\newtheorem{lemma}[theorem]{Lemma}

\newtheorem{proposition}[theorem]{Proposition}
\theoremstyle{definition}

\theoremstyle{remark}
\newtheorem{assumption}{Assumption}[section]
\newtheorem{remark}[theorem]{Remark}

\newcommand{\dist}{\mathsf d}
\newcommand{\Lip}{\operatorname{Lip}}
\newcommand{\Var}{\operatorname{Var}}
\newcommand{\diam}{\operatorname{diam}}
\newcommand{\Per}{\operatorname{Per}}

\newcommand{\E}{\mathbb{E}}
\newcommand{\N}{\mathbb{N}}

\newcommand{\K}{\mathrm{K}}

\renewcommand{\d}{{\mathrm d}}
\newcommand{\mm}{\mathfrak m}

\newcommand{\restr}[1]{\lower3pt\hbox{$|_{#1}$}}

\title{Quantitative Stability of Wasserstein Barycenters over Alexandrov Spaces with Lower Curvature Bounds}

\begin{document}
\author{Bang-Xian Han\thanks{School of Mathematics, Shandong University, Jinan, China. Email: hanbx@sdu.edu.cn.}
\and Zhuo-Nan Zhu\thanks{School of Mathematical Sciences, University of Science and Technology of China, Hefei, China. Email: zhuonanzhu@mail.ustc.edu.cn}}

\date{\today} 
\maketitle

\begin{abstract}
We prove quantitative stability estimates for Wasserstein barycenters on Alexandrov spaces with curvature bounded from below. The proof combines the variational strategy of Carlier--Delalande--M\'erigot with heat-kernel regularization, which supplies the regularity needed for dual convexity arguments in this non-smooth curved setting. The main result is an explicit strong-convexity modulus for the barycentric variance functional. As a consequence, barycenters depend H\"older-continuously on the underlying distributions with respect to the $1$-Wasserstein distance on the space of probability measures. We derive empirical-barycenter consistency and entropy-based sample-complexity bounds. Our proof does not rely on linear structure; in particular, the resulting estimates appear to be new even on smooth compact Riemannian manifolds.
\end{abstract}


\setcounter{tocdepth}{1}

\tableofcontents

\section{Introduction}

Wasserstein barycenters provide a geometrically intrinsic notion of averaging probability measures. Since the work of Agueh and Carlier~\cite{zbMATH05956494}, they have become a basic tool in optimal transport and its applications; see, for instance~\cite{Santambrogio2015,PeyreCuturi2019}. If $\mathbb{P}$ is a probability measure on $\mathcal{P}(\Omega)$, a barycenter of $\mathbb{P}$ is a minimizer of the variance functional
\[
\Var_{\mathbb{P}}(\mu)=\int_{\mathcal{P}(\Omega)} W_2^2(\mu,\rho)\,\d\mathbb{P}(\rho).
\]
Throughout the paper, 
\begin{equation}\label{KD}
W_2^2(\mu,\rho):=\sup_{\phi(x)+\psi(y)\leq c(x,y)}
\left\{\int_{\Omega}\phi\,\d\rho+\int_{\Omega}\psi\,\d\mu\right\},
\end{equation}
with $c(x,y)=\frac12\dist^2(x,y)$. This convention differs from the usual squared $L^2$-Wasserstein distance only by the harmless factor $1/2$. All Wasserstein distances below use this normalization. Optimal pairs $(\phi,\psi)$ in \eqref{KD}, with $\phi=\psi^c$, are called Kantorovich potentials.

Barycenters in Wasserstein space arise naturally in statistics, machine learning, image processing, shape analysis and partial differential equations; representative computational and applied developments include~\cite{RabinPeyreDelonBernot2012,CuturiDoucet2014,AltschulerBoixAdsera2021,PanaretosZemel2019,PeyreCuturi2019}. In many of these problems the ambient space is not a smooth Euclidean domain. Quotient spaces, shape spaces, stratified spaces, phylogenetic tree spaces and measured Gromov--Hausdorff limits all lead naturally to singular metric geometries. This motivates the study of Wasserstein barycenters over Alexandrov spaces with curvature bounded from below, a class that contains genuine singularities while retaining a rich first-order calculus and a well-developed optimal-transport theory~\cite{BuragoBuragoIvanov2001,zbMATH00606998,zbMATH05342782,zbMATH05349267,zbMATH06797407}.

Existence, uniqueness and absolute continuity of Wasserstein barycenters are well understood in Euclidean spaces~\cite{zbMATH05956494}, on Riemannian manifolds~\cite{KimPass2017}, over Alexandrov spaces~\cite{Ohta2012,zbMATH06797407}, and in more general metric-measure settings~\cite{LeGouicLoubes2017,arXiv:2412.01190}. Quantitative stability is subtler. One asks whether a perturbation of the second-order law $\mathbb{P}\in\mathcal{P}(\mathcal{P}(\Omega))$ forces a controlled perturbation of its barycenter. In Euclidean domains, Carlier, Delalande and M\'erigot~\cite{zbMATH07819904} proved a H\"older stability estimate by combining dual strong convexity, a global zero-order balance condition and interpolation estimates. Their argument relies essentially on linear and Euclidean structures, including the representation of optimal maps by gradients of convex functions and global McCann interpolation.

The main difficulty addressed in this paper is to replace these tools in a curved, non-smooth setting. On a smooth Riemannian manifold, optimal maps are written through the exponential map applied to gradients of $c$-concave potentials, and curvature and cut-locus effects enter the analysis. On Alexandrov spaces, smooth differential tools are unavailable, and even first-order objects must be handled through the metric calculus of semiconcave functions. Our strategy is to combine the variational framework of~\cite{zbMATH07819904} with the heat kernel regularization method developed by the authors in~\cite{arXiv:2602.19175}. The heat flow, whose analytic behavior on Alexandrov spaces is compatible with the curvature lower bound~\cite{GigliKuwadaOhta2013}, acts as a geometry-adapted mollifier for Kantorovich potentials. This allows us to recover the regularity needed for the dual convexity argument without relying on linear structure.

\paragraph{Contributions.}
The paper establishes a deterministic quantitative stability theory for Wasserstein barycenters over compact domains in Alexandrov spaces with arbitrary lower curvature bound. More precisely:
\begin{enumerate}[{\rm (i)}]
\item we prove a quantitative strong-convexity estimate for the single-target optimal-transport functional when the target is a good measure in the sense of Assumption~\ref{assump};
\item we derive a H\"older-type stability theorem for the barycenter map $\mathbb{P}\mapsto\mu_{\mathbb{P}}$ with respect to the $W_1$ distance on $\mathcal{P}(\mathcal{P}(\Omega))$;
\item we show how heat-kernel regularization and the Alexandrov first-order calculus yield a replacement for the Euclidean interpolation estimates used in~\cite{zbMATH07819904};
\item we develop statistical consequences, including empirical-barycenter consistency and sample-complexity estimates under metric entropy.
\end{enumerate}
To our knowledge, the resulting quantitative barycenter-stability estimates are new even for smooth Riemannian manifolds.

\subsection{Main quantitative stability results}\label{subsec:main-results}

Let $(X,\dist)$ be an $n$-dimensional Alexandrov space with curvature bounded from below by $k$, and let $\mathcal{H}^n$ be the $n$-dimensional Hausdorff measure. We fix a compact domain $\Omega\subseteq X$. For a probability measure $\nu$ and an integrable function $f$, we write $\E_\nu(f):=\int f\,\d\nu$. If $\mathbb{P},\mathbb{Q}\in\mathcal{P}(\mathcal{P}(\Omega))$, then $W_1(\mathbb{P},\mathbb{Q})$ denotes the $1$-Wasserstein distance on $\mathcal{P}(\mathcal{P}(\Omega))$ induced by the ground distance $W_2$ on $\mathcal{P}(\Omega)$.

We shall use the following uniform good-measure condition.

\begin{assumption}\label{assump}
There exist constants $\alpha_\mathbb{P}\in(0,1]$, $m_\mathbb{P},M_\mathbb{P},\eta_\mathbb{P},\Per_\mathbb{P}>0$ and a Borel set $S_\mathbb{P}\subseteq\mathcal{P}(\Omega)$ with $\mathbb{P}(S_\mathbb{P})\geq\alpha_\mathbb{P}$ such that every $\rho\in S_\mathbb{P}$ is a \emph{good measure} in the following uniform sense:
\begin{itemize}
\item $\rho$ is concentrated on a bounded open John domain $S_\rho\subseteq\Omega$ with parameter $\eta_\mathbb{P}$, namely $\rho(\Omega\setminus S_\rho)=0$ and there is a point $x_0\in S_\rho$ such that, for every $x\in S_\rho$, there exists a rectifiable curve $\gamma:[0,\ell(\gamma)]\to S_\rho$, parametrized by arclength from $x$ to $x_0$, satisfying
\[
\dist(\gamma(t),S_\rho^c)\geq \eta_\mathbb{P}t,\qquad t\in[0,\ell(\gamma)].
\]
\item $S_\rho$ has finite perimeter in $X$ and $\Per(S_\rho)\leq\Per_\mathbb{P}$;
\item as measures on $\Omega$,
\[
 m_\mathbb{P}\mathcal{H}^n\llcorner S_\rho\leq \rho\leq M_\mathbb{P}\mathcal{H}^n\llcorner S_\rho.
\]
\end{itemize}
\end{assumption}

\begin{remark}[Role of the good-measure assumption]
Assumption~\ref{assump} is imposed only on a positive-mass subfamily of target measures. Thus the reference law $\mathbb{P}$ may contain singular, discrete, or otherwise irregular targets outside $S_\mathbb{P}$. The uniformity on $S_\mathbb{P}$ is used to obtain a quantitative dual convexity modulus, while the positive mass $\alpha_\mathbb{P}$ transfers this single-target convexity to the barycentric variance functional.
\end{remark}

The first theorem is a strong-convexity estimate for a single good target measure.
\begin{theorem}\label{main}
Let $\rho\in\mathcal{P}(\Omega)$ be a good measure with parameters $m,M,\eta,\Per$. Then for any $\mu_0,\mu_1\in\mathcal{P}(\Omega)$ and any Kantorovich potential $\psi_0$ on the $\mu_0$-side, i.e. any optimal pair $(\psi_0^c,\psi_0)$ for $(\rho,\mu_0)$, one has
\begin{equation}\label{eq:main_modulus}
W_2^2(\mu_1,\rho)-W_2^2(\mu_0,\rho)-\int_\Omega \psi_0\,\d(\mu_1-\mu_0)
\geq \mathcal{D}\bigl(W_2(\mu_0,\mu_1)\bigr),
\end{equation}
where, with $D_W:=\diam(\mathcal{P}(\Omega),W_2)$,
\[
\mathcal{D}(t):=\frac{A_1t^{12}}{A_2+A_3\left|\log(t/D_W)\right|}\quad(0<t\leq D_W),\qquad \mathcal{D}(0):=0.
\]
The constants $A_1,A_2,A_3>0$ depend only on $k,n,m,M,\eta,\Per$ and $\diam(\Omega)$.
\end{theorem}

The second theorem gives quantitative stability of Wasserstein barycenters. By~\cite[Theorem~4.1]{zbMATH06797407}, every $\mathbb{P}$ satisfying Assumption~\ref{assump} has a unique Wasserstein barycenter $\mu_\mathbb{P}$. A general $\mathbb{Q}\in\mathcal{P}(\mathcal{P}(\Omega))$ need not have a unique barycenter.
\begin{theorem}\label{barycenter}
Let $\sigma>0$, and let $\mathbb{P}\in\mathcal{P}(\mathcal{P}(\Omega))$ satisfy Assumption~\ref{assump}. Then there exists a constant $C>0$, depending only on $\sigma$, $k,n,\alpha_\mathbb{P},m_\mathbb{P},M_\mathbb{P},\eta_\mathbb{P},\Per_\mathbb{P}$ and $\diam(\Omega)$, such that for every $\mathbb{Q}\in\mathcal{P}(\mathcal{P}(\Omega))$ and every barycenter $\mu_\mathbb{Q}$ of $\mathbb{Q}$,
\begin{equation}\label{eq:stability_estimate}
C\,W_2^{12+\sigma}(\mu_\mathbb{P},\mu_\mathbb{Q})\leq W_1(\mathbb{P},\mathbb{Q}).
\end{equation}
\end{theorem}

The power $12$ comes from two effects: the heat-regularized lower bound is quadratic in a centered $L^1$ oscillation, while the Alexandrov stability estimate for optimal maps converts the $L^1$ distance of dual potentials into the sixth power of $W_2(\mu_0,\mu_1)$. The logarithmic loss in Theorem~\ref{main} is harmless for Theorem~\ref{barycenter}, since the modulus $\mathcal{D}$ dominates $t^{12+\sigma}$ on bounded intervals for every fixed $\sigma>0$.

\subsection{Statistical consequences and comparison with related work}\label{subsec:stats-overview}

Section~\ref{sec:statistical-applications} derives several consequences of Theorem~\ref{barycenter}. If $\mathbb{P}_N=N^{-1}\sum_{i=1}^N\delta_{\rho_i}$ is the empirical law of i.i.d. samples from $\mathbb{P}$, then every empirical barycenter $\mu_{\mathbb{P}_N}$ is a natural estimator of $\mu_\mathbb{P}$. Under a polynomial metric-entropy assumption on $K=\operatorname{spt}\mathbb{P}\subset(\mathcal{P}(\Omega),W_2)$, we obtain the high-probability sample-complexity bound
\[
N\gtrsim \varepsilon^{-(12+\sigma)(d_{\rm ent}+2)}
      +\varepsilon^{-2(12+\sigma)}\log\frac1\delta
\quad\Longrightarrow\quad
\mathsf P\bigl(W_2(\mu_{\mathbb{P}_N},\mu_\mathbb{P})\le \varepsilon\bigr)\ge 1-\delta.
\]
The polynomial covering-number assumption is a low-complexity condition on $K$; separately, for the full Wasserstein space we prove an exponential metric-entropy estimate over a compact Alexandrov domain. This yields the logarithmic expected rate
\[
\mathbb E W_2(\mu_{\mathbb P_N},\mu_{\mathbb P})
\lesssim
\left(\frac{\log\log N}{\log N}\right)^{1/(n(12+\sigma))}.
\]

The empirical theory of barycenters in general metric spaces has been developed from several complementary viewpoints. Le Gouic and Loubes~\cite{LeGouicLoubes2017} proved existence and consistency in broad metric settings. Ahidar-Coutrix, Le Gouic and Paris~\cite{AhidarLeGouicParis2020} obtained rates under a variance inequality, verifying it in non-positively curved Alexandrov spaces and in certain non-negatively curved spaces with extendable geodesics. Le Gouic, Paris, Rigollet and Stromme~\cite{LeGouicParisRigolletStromme2023} proved fast, including parametric, rates under bi-extendibility or hugging assumptions. The present paper addresses a different deterministic stability problem: it estimates the variation of the barycenter map under perturbations of the second-order law. Our ambient space is an {\bf arbitrary compact domain} in an Alexandrov space with curvature bounded from below; we do not impose non-positive curvature or geodesic extendibility. The price is the uniform good-measure assumption, which is precisely what makes the heat-kernel dual-convexity argument quantitative.

\paragraph{Open problems.}
The results suggest several directions for further work. A first problem is to extend the quantitative stability theory from Alexandrov spaces to more general metric measure spaces satisfying the $\mathrm{RCD}(K,N)$ condition. A second problem is to optimize the dependence of the stability exponent and the constants on the geometric dimension. Dimension-free estimates, if available in suitable infinite-dimensional settings such as Wiener space or configuration spaces (cf.~\cite{arXiv:2412.01190}), would be particularly relevant for high-dimensional probability and statistics. Finally, in the Euclidean setting Carlier--Delalande--M\'{e}rigot~\cite[Section~1.5]{zbMATH07819904} obtain a stronger linear stability mechanism leading to
\[
W_2(\mu_{\mathbb{P}_N},\mu_\mathbb{P})\lesssim W_1^{1/2}(\mathbb{P}_N,\mathbb{P}).
\]
Even in the entropy-favorable case $d_{\rm ent}=0$, our exponent gives the slower rate $N^{-1/(24+2\sigma)}$ up to logarithmic factors. Understanding whether the exponent $12$ is intrinsic or an artifact of the heat-regularization argument remains an important question.

\paragraph{Organization.}
Section~\ref{sec:main-proof} proves Theorems~\ref{main} and~\ref{barycenter}. The proof consists of heat-kernel regularization, a lower bound through regularized Kantorovich functionals, a limit analysis, and an Alexandrov stability estimate for optimal maps. Section~\ref{sec:statistical-applications} contains the statistical applications. The proof of the global zero-order balance condition is deferred to Appendix~\ref{A}.

\section{Quantitative Stability Theorems}\label{sec:main-proof}

\subsection{Heat kernel regularization}\label{sub}
We briefly recall the heat-kernel regularization framework developed in \cite[Section 2]{arXiv:2602.19175}. 

Throughout this subsection, a Kantorovich potential originally defined on $\Omega$ is extended to $X$ by a fixed McShane extension, and, when necessary, by the harmless truncation used in \cite[Section~2]{arXiv:2602.19175} to meet the decay hypotheses. The extension preserves the Lipschitz constant. All estimates below depend only on that Lipschitz constant and on the structural parameters, and the limiting quantities in the arguments are independent of the chosen extension; see \cite[Lemmas~2.12--2.13]{arXiv:2602.19175}.

Let $p_t(x,y)$, $t>0$, be the heat kernel, and let $\psi\in\Lip(X,\dist)$ have the decay required in \cite[Section~2]{arXiv:2602.19175}. We define the heat-kernel-regularized $c$-transform by
\begin{equation*}
\Phi_t[\psi](x):=-t \log \int_X e^{\frac{{\psi}(y)}{t}} p_{\frac{t}{2}}(x, y)\, \d\mm(y).
\end{equation*}
Given a good measure $\rho$, we define the heat-kernel-regularized Kantorovich functional 
\begin{equation*}
\K_t[\psi]:=\int_{S_\rho} \Phi_t[\psi]\,\d\rho.
\end{equation*}

For $x\in X$, we define  
\begin{equation*}
\d\mu_x^t[\psi](y):=\frac{e^{\psi(y)/t}p_{t/2}(x,y)\,\d\mm(y)}{\int_X e^{\psi(y)/t}p_{t/2}(x,y)\,\d\mm(y)},\qquad
\mu^t[\psi]:=\int_{S_\rho}\mu_x^t[\psi]\,\d\rho(x).
\end{equation*}

\begin{lemma}[{\cite[Lemma 2.3, Proposition 2.10]{arXiv:2602.19175}}]\label{l3.2}
Let $\rho$ be a good measure with parameters $m,M,\eta,\Per$. Let $\psi_0,\psi_1\in\Lip(X,\dist)$ be functions with the required fast decay at infinity. Denote $v=\psi_1-\psi_0$ and $\psi_s=\psi_0+s v$, $s\in [0,1]$. For $t$ small enough, it holds 
\begin{itemize}
\item $\frac{\d}{\d s}{\K}_t[\psi_s]=-  \E_{\mu^t[\psi_s]}(v)$;
\item $\frac{\d^2}{\d s^2}{\K}_t[\psi_s]=-\frac{1}{t}\int_{S_\rho} {\rm{Var}}_{ \mu_x^t[\psi_s]}(v)\,\d\rho(x)$;
\item there is $\kappa=\kappa(k,n,\eta,m,M,\Per,\diam(\Omega))$ such that
\begin{equation*}
\int_{S_\rho}\big| \E_{\mu_x^t[\psi_s]}(v)-\E_{\mu^t[\psi_s]}(v)\big|\,\d\rho(x)\leq  \frac{\kappa}{\sqrt{t}}\left(\int_{S_\rho} {\rm{Var}}_{ \mu_x^t[\psi_s]}(v)\,\d\rho(x)\right)^\frac{1}{2}.
\end{equation*}
\end{itemize}
\end{lemma}

\begin{theorem}[Stability of Kantorovich potentials; {\cite[Theorem~3.1]{arXiv:2602.19175}}]
\label{thm:quoted-potential-stability}
Let $(X,\dist)$ be an $n$-dimensional Alexandrov space with curvature bounded from below by $k$, let $\Omega\subset X$ be compact, and let $\rho$ be concentrated on a good John domain $S_\rho\subseteq\Omega$ with parameters $m,M,\eta,\Per$. Let $(\phi_i,\psi_i)$, $i=0,1$, be Kantorovich pairs for $(\rho,\mu_i)$, with $\phi_i=\psi_i^c$, and assume the normalization $\E_\rho(\phi_i)=0$. Then the gradients $\nabla\phi_i$ are defined $\rho$-a.e. on $S_\rho$, and there exists a constant $C_{\rm pot}>0$, depending only on $k,n,m,M,\eta,\Per$ and $\diam(\Omega)$, such that
\[
\int_{S_\rho}|\nabla\phi_1-\nabla\phi_0|^2\,\d\rho
\leq C_{\rm pot}
\left(\int_{S_\rho}|\phi_1-\phi_0|^2\,\d\rho\right)^{1/3}.
\]
The estimate is invariant under adding constants to the potentials once the above normalization is imposed.
\end{theorem}

\subsection{Lower bound via regularized functionals}
Denote 
\[
D_\rho(\mu_1,\mu_0):=W_2^2(\mu_1, \rho)-W_2^2(\mu_0, \rho)-\int_\Omega \psi_0\,\d (\mu_1-\mu_0).
\]

\begin{lemma}\label{3.3}
Assume that $\psi_0$ is a Kantorovich potential from $\mu_0$ to $\rho$. Then $D_\rho(\mu_1,\mu_0)\geq 0$.
\end{lemma}
\begin{proof}
By Kantorovich duality, it holds
\begin{equation}
W_2^2(\mu_0, \rho)= \int_\Omega \psi_0\,\d\mu_0+\int_{S_\rho} \psi_0^c\,\d\rho.
\end{equation}
Then 
\begin{equation}
D_\rho(\mu_1,\mu_0)=W_2^2(\mu_1, \rho)-\left(\int_\Omega \psi_0\,\d\mu_1+\int_{S_\rho} \psi_0^c\,\d\rho\right)\geq 0.
\end{equation}
\end{proof}

\begin{lemma}\label{l3.4}
Assume that, for $i=0,1$, $\psi_i$ is a Kantorovich potential from $\mu_i$ to $\rho$. It holds that
\begin{equation*}
D_\rho(\mu_1,\mu_0)=\lim_{t\rightarrow0}\int_0^1 -s \frac{\d^2}{\d s^2}{\K}_t[\psi_s]\,\d s.
\end{equation*}
\end{lemma}
\begin{proof}
Note that 
\begin{equation}
\K_t[\psi_1]-\K_t[\psi_0]+ \E_{\mu^t[\psi_1]}(v)=\int_0^1 -s \frac{\d^2}{\d s^2}{\K}_t[\psi_s]\,\d s.
\end{equation}
By \cite[Lemma 2.12, Lemma 2.13]{arXiv:2602.19175}, letting $t\rightarrow0$, we have  
\begin{equation}\label{1.7}
\begin{aligned}
\lim_{t\rightarrow0}\int_0^1 -s \frac{\d^2}{\d s^2}{\K}_t[\psi_s]\,\d s
&=\lim_{t\rightarrow0}\bigl(\K_t[\psi_1]-\K_t[\psi_0]+ \E_{\mu^t[\psi_1]}(v)\bigr)\\
&= \int_{S_\rho}\psi_1^c\,\d\rho- \int_{S_\rho}\psi_0^c\,\d\rho+\int_\Omega (\psi_1-\psi_0) \,\d\mu_1.
\end{aligned}
\end{equation}
Then the lemma follows from the Kantorovich duality.
\end{proof}

\begin{proposition}\label{3.22}
It holds that
\begin{equation*}
D_\rho(\mu_1,\mu_0)\geq \frac{1}{\kappa^2}\int_0^1 s \varliminf_{t\rightarrow0}\|\E_{\mu_x^t[\psi_s]}(v)-\E_{\mu^t[\psi_s]}(v)\|_{L^1(\rho)}^2\,\d s.
\end{equation*}
\end{proposition}
\begin{proof}
It follows from Lemma~\ref{l3.2}, Lemma~\ref{l3.4} and Fatou's lemma.
\end{proof}

\subsection{Limit analysis}
Denote 
\[
\phi_s(x):=\inf_{y\in \Omega}\big\{c(x,y)-\psi_s(y)\big\}=\psi_s^c(x), \quad x\in S_{\rho}.
\]

\begin{lemma}\label{mini}
For $\rho$-a.e.\ $x\in S_\rho$, the function $y\mapsto c(x,y)-\psi_s(y)$ has a unique minimizer $T_s(x)$.
\end{lemma}
\begin{proof}
Since $\Omega$ is compact and $c(x,y), \psi_s(y)$ are continuous in $y\in \Omega$, for any $x\in S_{\rho}$, there exists a minimizer $y_x$ so that $\psi_s^c(x)+\psi_s(y_x)=c(x,y_x)$.
Note that $\psi_s(y)\leq c(x,y)-\psi_s^c(x)$. 
We have $\psi_s(y)\leq \psi_s^{\mathrm{cc}}(y)$
and 
\[
\psi_s^{\mathrm{cc}}(y_x)\leq c(x,y_x)-\psi_s^c(x)=\psi_s(y_x).
\]
Thus $\psi^{\mathrm{cc}}_s(y_x)=\psi_s(y_x)$ and it holds 
\begin{equation}
\phi_s(x)+\phi_s^c(y_x)=\psi_s^c(x)+\psi^{\mathrm{cc}}_s(y_x)=\psi_s^c(x)+\psi_s(y_x)=c(x,y_x),
\end{equation}
which implies $y_x\in \partial^c\phi_s(x)$. 

Note that $\phi_s$ is $c$-concave and Lipschitz. By Rademacher's theorem (cf.\ \cite[Corollary~2.14]{zbMATH05349267}), $\phi_s$ is differentiable at $\rho$-a.e.\ $x\in S_\rho$, and by \cite[Lemma~3.5]{zbMATH05342782},
\[
T_s(x):=\exp_x(-\nabla\phi_s(x)),\qquad \partial^c\phi_s(x)=\{T_s(x)\}.
\]
This proves the claim.
\end{proof}

\begin{lemma}\label{limit}
For $s\in [0,1]$, it holds that 
\begin{equation*}
\lim_{t\rightarrow0}\E_{\mu_x^t[\psi_s]}(v)=v(T_s(x)),\quad\text{for $\rho$-a.e.\ $x\in S_{\rho}$}.
\end{equation*}
Moreover, we have
\begin{equation*}
\lim_{t\rightarrow0}\E_{\mu^t[\psi_s]}(v)=\E_{\rho}(v\circ T_s). 
\end{equation*}
\end{lemma}
\begin{proof}
This follows directly by combining Lemma~\ref{mini} and \cite[Lemma 2.13]{arXiv:2602.19175}.     
\end{proof}

\begin{proposition}\label{limits}
It holds that
\begin{equation*}
D_\rho(\mu_1,\mu_0)\geq \frac{1}{\kappa^2}\int_0^1 s \|v\circ T_s-\E_\rho(v\circ T_s)\|_{L^1(\rho)}^2\,\d s.
\end{equation*}
\end{proposition}
\begin{proof}
It follows from Proposition~\ref{3.22}, Lemma~\ref{limit} and the dominated convergence theorem. 
\end{proof}

\subsection{Stability of optimal transport}

\begin{lemma}\label{3.7}
It holds that
\[
\frac{\d}{\d s} \phi_s(x)=-v(T_s(x)), \quad\text{for $\mathcal{L}^1\otimes\rho$-a.e. $(s,x)\in [0,1]\times S_{\rho}$},
\]
where at the endpoints $s=0,1$ the derivative is understood in the one-sided sense.
\end{lemma}

\begin{proof}
By Lemma~\ref{mini} and Fubini's theorem, for $\mathcal{L}^1\otimes\rho$-a.e. $(s,x)$ the minimizer $T_s(x)$ is unique. Fix such a pair $(s,x)$. For $s_1,s_2\in[0,1]$ such that $T_{s_2}(x)$ is well-defined, the definition of $\phi_{s_1}$ gives
\begin{equation}\label{2.6}
\phi_{s_1}(x)\leq c(x,T_{s_2}(x))-\psi_{s_1}(T_{s_2}(x))
=\phi_{s_2}(x)+(s_2-s_1)v(T_{s_2}(x)).
\end{equation}
Taking $s_1=s+h$ and $s_2=s$ in \eqref{2.6} yields, for $h>0$,
\begin{equation}\label{29}
\varlimsup_{h\downarrow0}\frac{\phi_{s+h}(x)-\phi_s(x)}{h}\leq -v(T_s(x)).
\end{equation}
Conversely, taking $s_1=s$ and $s_2=s+h$ gives
\[
\frac{\phi_{s+h}(x)-\phi_s(x)}{h}\geq -v(T_{s+h}(x)).
\]
If $h_j\downarrow0$ and $T_{s+h_j}(x)$ converges along a subsequence to $y$, compactness of $\Omega$ and continuity of $c$ and $\psi_s$ imply that $y$ is a minimizer of $c(x,\cdot)-\psi_s$. By uniqueness, $y=T_s(x)$. Hence $v(T_{s+h_j}(x))\to v(T_s(x))$ along every convergent subsequence, and therefore
\begin{equation}\label{31}
\varliminf_{h\downarrow0}\frac{\phi_{s+h}(x)-\phi_s(x)}{h}\geq -v(T_s(x)).
\end{equation}
Combining \eqref{29} and \eqref{31} gives the right derivative. The left derivative is obtained in the same way, and the endpoint statements follow by the corresponding one-sided argument.
\end{proof}

\begin{lemma}[Lipschitz control of the Alexandrov exponential map]
\label{lem:alex-exp-comparison}
Let $X$ be an Alexandrov space with curvature bounded below by $k$, and let $D>0$. There exists $C_{\exp}=C_{\exp}(k,D)>0$ such that, for every $x\in X$ and every pair of tangent vectors $u,v\in T_xX$ with $|u|,|v|\leq D$ for which the endpoints $\exp_x(u)$ and $\exp_x(v)$ are defined along minimizing geodesics,
\[
\dist(\exp_x(u),\exp_x(v))\leq C_{\exp}|u-v|_{T_xX}.
\]
Consequently, if $\phi_0,\phi_1$ are $c$-concave potentials whose associated optimal maps are $T_i(x)=\exp_x(-\nabla\phi_i(x))$ and whose images lie in $\Omega$, then for a.e. such $x$,
\[
\dist(T_1(x),T_0(x))\leq C_{\exp}|\nabla\phi_1(x)-\nabla\phi_0(x)|_{T_xX},
\]
where $C_{\exp}$ depends only on $k$ and $\diam(\Omega)$.
\end{lemma}

\begin{proof}
This is a standard consequence of the hinge, triangle, and angle comparison estimates in Alexandrov geometry. Toponogov comparison bounds the side opposite the angle between two geodesics from $x$ by the corresponding side in the two-dimensional model space of curvature $k$. On the compact range $|u|,|v|\leq D$, the model-space law of cosines is Lipschitz with respect to the tangent-cone distance $|u-v|_{T_xX}$, which gives the asserted constant; see, for example, \cite[Chapter~10]{BuragoBuragoIvanov2001} and \cite{zbMATH05342782,zbMATH06032507}. The final statement follows by applying this estimate to $u=-\nabla\phi_0(x)$ and $v=-\nabla\phi_1(x)$ at points where the potentials are differentiable and the optimal maps are represented by the Alexandrov exponential map.
\end{proof}

\begin{proposition}\label{gs}
Assume that $(\phi_i,\psi_i)$, $i=0,1$, are Kantorovich pairs for $(\rho,\mu_i)$, with $\phi_i=\psi_i^c$, and that $\E_\rho(\phi_i)=0$ for $i=0,1$. Denote
\[
g(s)=\|v\circ T_s-\E_\rho(v\circ T_s)\|_{L^1(\rho)}.
\]
Then
\[
\int_0^1 g(s)\,\d s\geq C_2 W_2^6(\mu_0,\mu_1),
\]
where $C_2>0$ depends only on $k,n,\eta,m,M,\Per$ and $\diam(\Omega)$.
\end{proposition}
\begin{proof}
By Lemma~\ref{3.7} and dominated convergence,
\begin{equation}\label{ds}
\frac{\d}{\d s}\bigl(\phi_s(x)-\E_\rho(\phi_s)\bigr)
= -\bigl(v(T_s(x))-\E_\rho(v\circ T_s)\bigr)
\end{equation}
for $\mathcal L^1\otimes\rho$-a.e. $(s,x)$. Integrating \eqref{ds} in $s$ and using $\E_\rho(\phi_i)=0$ gives
\begin{equation}\label{in}
\int_0^1 \bigl(v(T_s(x))-\E_\rho(v\circ T_s)\bigr)\,\d s=\phi_0(x)-\phi_1(x).
\end{equation}
Therefore, by Fubini's theorem and the triangle inequality,
\begin{equation}\label{g1}
\int_0^1 g(s)\,\d s
\geq \int_{S_\rho}|\phi_1(x)-\phi_0(x)|\,\d\rho(x).
\end{equation}

Set $f:=\phi_1-\phi_0$. The functions $\phi_i$ are uniformly Lipschitz and normalized by $\E_\rho(\phi_i)=0$, hence $\|f\|_{L^\infty}$ is bounded by a constant depending only on $\diam(\Omega)$. The quoted potential stability theorem, Theorem~\ref{thm:quoted-potential-stability}, gives
\[
\int_{S_\rho}|\nabla\phi_1-\nabla\phi_0|^2\,\d\rho
\leq C\left(\int_{S_\rho}|f|^2\,\d\rho\right)^{1/3}.
\]
Using the $L^\infty$ bound on $f$, we obtain
\begin{equation}\label{eq:L1_grad}
\int_{S_\rho}|f|\,\d\rho
\geq c\left(\int_{S_\rho}|\nabla\phi_1-\nabla\phi_0|^2\,\d\rho\right)^3.
\end{equation}
For $\rho$-a.e. $x$, the gradients $-\nabla\phi_i(x)$ are well-defined~\cite{zbMATH05342782}. Since the endpoints $T_i(x)$ lie in $\Omega$, Lemma~\ref{lem:alex-exp-comparison} gives
\begin{equation}\label{exp}
\dist(T_1(x),T_0(x))\leq C|\nabla\phi_1(x)-\nabla\phi_0(x)|,
\end{equation}
with $C$ depending only on the curvature lower bound and $\diam(\Omega)$. Since $(T_0,T_1)_\#\rho$ is a coupling of $\mu_0$ and $\mu_1$, our normalization $c=\frac12\dist^2$ gives
\[
2 W_2^2(\mu_0,\mu_1)\leq \int_{S_\rho}\dist^2(T_0(x),T_1(x))\,\d\rho(x).
\]
Combining this inequality with \eqref{eq:L1_grad} and \eqref{exp} yields
\begin{equation}\label{g2}
\int_{S_\rho}|\phi_1-\phi_0|\,\d\rho
\geq C_2 W_2^6(\mu_0,\mu_1).
\end{equation}
The conclusion follows from \eqref{g1} and \eqref{g2}.
\end{proof}

\medskip

\begin{proof}[Proof of Theorem~\ref{main}]
Choose a Kantorovich potential $\psi_1$ from $\mu_1$ and set $\phi_i:=\psi_i^c$. Adding constants to $\psi_0$ and $\psi_1$ does not change the left-hand side of \eqref{eq:main_modulus}; it only adds constants to $v=\psi_1-\psi_0$, which disappear in the centered quantity defining $g$. We may therefore normalize $\E_\rho(\phi_i)=0$ for $i=0,1$.

By Proposition~\ref{limits},
\begin{equation}\label{eq:D_lower_g}
D_\rho(\mu_1,\mu_0)\geq \frac1{\kappa^2}\int_0^1 s g^2(s)\,\d s,
\qquad
g(s)=\|v\circ T_s-\E_\rho(v\circ T_s)\|_{L^1(\rho)}.
\end{equation}
The function $g$ is bounded by a constant $G$ depending only on $\diam(\Omega)$. Let $R:=W_2(\mu_0,\mu_1)$. If $R=0$, there is nothing to prove. Otherwise set
\[
\delta:=\min\left\{\frac12,\frac{C_2R^6}{2G}\right\}.
\]
Then $\int_0^\delta g(s)\,\d s\leq C_2R^6/2$. Proposition~\ref{gs} gives
\[
\int_\delta^1 g(s)\,\d s\geq \frac{C_2}{2}R^6.
\]
By Cauchy's inequality with weight $s$ and \eqref{eq:D_lower_g},
\[
\frac{C_2}{2}R^6
\leq \left(\int_\delta^1\frac{\d s}{s}\right)^{1/2}
     \left(\int_\delta^1 s g^2(s)\,\d s\right)^{1/2}
\leq \kappa |\log\delta|^{1/2}D_\rho(\mu_1,\mu_0)^{1/2}.
\]
Hence
\[
D_\rho(\mu_1,\mu_0)\geq c\frac{R^{12}}{|\log\delta|}.
\]
Let $D_W:=\diam(\mathcal{P}(\Omega),W_2)$. Since $0<R\leq D_W$ and $\delta$ is comparable to $(R/D_W)^6$ for small $R/D_W$ while bounded below by a positive constant for large $R/D_W$, there exist constants $A_1,A_2,A_3>0$, depending only on the stated structural parameters, such that
\[
D_\rho(\mu_1,\mu_0)\geq \frac{A_1R^{12}}{A_2+A_3\left|\log(R/D_W)\right|}.
\]
This is \eqref{eq:main_modulus}.
\end{proof}

\medskip

\begin{proof}[Proof of Theorem~\ref{barycenter}]
Let $(\phi_\rho,\psi_\rho)$ be the measurable family of Kantorovich potentials given by Proposition~\ref{balance}, so that $(\phi_\rho,\psi_\rho)$ is optimal for $(\rho,\mu_\mathbb{P})$ and
\[
\int_{\mathcal{P}(\Omega)}\psi_\rho(y)\,\d\mathbb{P}(\rho)=0,
\qquad y\in\Omega.
\]
For every $\rho\in S_\mathbb{P}$, Theorem~\ref{main}, applied with $\mu_0=\mu_\mathbb{P}$ and $\mu_1=\mu_\mathbb{Q}$, gives the uniform lower bound
\begin{equation}\label{eq:pointwise_deficit}
W_2^2(\mu_\mathbb{Q},\rho)-W_2^2(\mu_\mathbb{P},\rho)
-\int_\Omega \psi_\rho\,\d(\mu_\mathbb{Q}-\mu_\mathbb{P})
\geq \mathcal{D}\bigl(W_2(\mu_\mathbb{P},\mu_\mathbb{Q})\bigr).
\end{equation}
For $\rho\notin S_\mathbb{P}$, the same left-hand side is nonnegative by Lemma~\ref{3.3}. Integrating in $\rho$ and using the balance condition, we obtain
\begin{equation}\label{eq:variance_lower}
\Var_\mathbb{P}(\mu_\mathbb{Q})-\Var_\mathbb{P}(\mu_\mathbb{P})
\geq \alpha_\mathbb{P}\mathcal{D}\bigl(W_2(\mu_\mathbb{P},\mu_\mathbb{Q})\bigr).
\end{equation}
Indeed, the term involving $\psi_\rho$ vanishes by Fubini's theorem:
\[
\int_{\mathcal{P}(\Omega)}\int_\Omega \psi_\rho\,\d(\mu_\mathbb{Q}-\mu_\mathbb{P})\,\d\mathbb{P}(\rho)=0.
\]

On the other hand, since $\mu_\mathbb{Q}$ is a barycenter of $\mathbb{Q}$,   $\Var_\mathbb{Q}(\mu_\mathbb{Q}) \leq \Var_\mathbb{Q}(\mu_\mathbb{P})$,
\begin{align*}
\Var_\mathbb{P}(\mu_\mathbb{Q})-\Var_\mathbb{P}(\mu_\mathbb{P})
&=\bigl[\Var_\mathbb{P}(\mu_\mathbb{Q})-\Var_\mathbb{Q}(\mu_\mathbb{Q})\bigr]\\
&\quad +\bigl[\Var_\mathbb{Q}(\mu_\mathbb{Q})-\Var_\mathbb{Q}(\mu_\mathbb{P})\bigr]\\
&\quad +\bigl[\Var_\mathbb{Q}(\mu_\mathbb{P})-\Var_\mathbb{P}(\mu_\mathbb{P})\bigr]\\
&\leq \left|\int_{\mathcal{P}(\Omega)} W_2^2(\mu_\mathbb{Q},\rho)\,\d(\mathbb{P}-\mathbb{Q})(\rho)\right|\\
&\quad +\left|\int_{\mathcal{P}(\Omega)} W_2^2(\mu_\mathbb{P},\rho)\,\d(\mathbb{P}-\mathbb{Q})(\rho)\right|.
\end{align*}
The maps $\rho\mapsto W_2^2(\mu_\mathbb{Q},\rho)$ and $\rho\mapsto W_2^2(\mu_\mathbb{P},\rho)$ are $L_\Omega$-Lipschitz on $(\mathcal{P}(\Omega),W_2)$, with $L_\Omega$ depending only on $\diam(\Omega)$. Hence Kantorovich duality for $W_1$ gives
\begin{equation}\label{eq:D_to_W1}
B_1\mathcal{D}\bigl(W_2(\mu_\mathbb{P},\mu_\mathbb{Q})\bigr)
\leq W_1(\mathbb{P},\mathbb{Q})
\end{equation}
for some $B_1>0$ depending only on $\alpha_\mathbb{P}$ and $\diam(\Omega)$.

Finally, $W_2(\mu_\mathbb{P},\mu_\mathbb{Q})$ is bounded above by $\diam(\mathcal{P}(\Omega),W_2)$. Therefore, for every $\sigma>0$ there exists $c_\sigma>0$ such that
\[
\mathcal{D}(t)\geq c_\sigma t^{12+\sigma},\qquad 0\leq t\leq\diam(\mathcal{P}(\Omega),W_2).
\]
Combining this with \eqref{eq:D_to_W1} proves \eqref{eq:stability_estimate}.
\end{proof}

\section{Statistical Applications}\label{sec:statistical-applications}

The results in this section are consequences of the deterministic stability estimate in Theorem~\ref{barycenter}. We first treat empirical barycenters under a polynomial covering-number condition, then record a separate geometric entropy bound for the full Wasserstein space over a compact Alexandrov space, and finally prove almost-sure consistency.

Throughout this section, random variables are defined on a probability space $(\Theta,\mathcal F,\mathsf P)$. The symbol $\mathsf P$ denotes this underlying probability measure, whereas $\mathbb P$ denotes a probability law on the Wasserstein space $\mathcal P(\Omega)$. Let $\rho_1,\dots,\rho_N$ be i.i.d. random variables with law $\mathbb{P}$, and set
\[
\mathbb{P}_N:=\frac1N\sum_{i=1}^N\delta_{\rho_i}\in\mathcal{P}(\mathcal{P}(\Omega)).
\]
The empirical barycenter of $\mathbb{P}_N$ is a natural estimator of $\mu_\mathbb{P}$.

For a metric space $(Y,d_Y)$, write $\mathcal N(Y,d_Y,r)$ for the minimal number of open $d_Y$-balls of radius $r$ needed to cover $Y$. If the metric is clear, we write simply $\mathcal N(Y,r)$.

\subsection{Sample complexity under polynomial covering numbers}
Since $\mathcal{P}(\Omega)$ is typically infinite-dimensional, non-asymptotic rates for $W_1(\mathbb{P}_N,\mathbb{P})$ require additional complexity assumptions. We first impose a polynomial covering-number condition on the support of $\mathbb P$. Let
\[
\mathsf Z:=(\mathcal{P}(\Omega),W_2),\qquad \Delta:=\diam(\mathsf Z).
\]

\begin{proposition}[Sample complexity under polynomial covering numbers]\label{prop:sample}
Assume that $\mathbb{P}\in\mathcal{P}(\mathcal{P}(\Omega))$ satisfies Assumption~\ref{assump}. Let $K:=\operatorname{spt}\mathbb{P}\subset\mathsf Z$, and suppose that for some constants $C_{\rm ent}\geq1$ and $d_{\rm ent}\geq0$,
\begin{equation}\label{eq:entropy_assumption}
\mathcal N(K,W_2,r)\leq C_{\rm ent}r^{-d_{\rm ent}},\qquad 0<r\leq1.
\end{equation}
Then for every $\sigma>0$, $\varepsilon\in(0,1/2)$ and $\delta\in(0,1)$, there exists $C>0$, depending only on $\sigma$, the structural parameters in Assumption~\ref{assump}, $\Delta$, $C_{\rm ent}$ and $d_{\rm ent}$, such that if
\begin{equation}\label{eq:sample_complexity}
N\geq C\left(\varepsilon^{-(12+\sigma)(d_{\rm ent}+2)}+
\varepsilon^{-2(12+\sigma)}\log\!\frac{1}{\delta}\right),
\end{equation}
then any empirical Wasserstein barycenter $\mu_{\mathbb{P}_N}$ satisfies
\[
\mathsf P\!\left(W_2(\mu_{\mathbb{P}_N},\mu_\mathbb{P})\leq\varepsilon\right)\geq1-\delta.
\]
\end{proposition}

\begin{proof}
Set $q_\sigma:=12+\sigma$. By Theorem~\ref{barycenter}, applied with the population law $\mathbb{P}$ as the reference distribution and $\mathbb{P}_N$ as the perturbation, there exists $c_*>0$ depending only on the stated structural constants such that
\begin{equation}\label{eq:empirical_stability}
c_* W_2^{q_\sigma}(\mu_{\mathbb{P}_N},\mu_\mathbb{P})
\leq W_1(\mathbb{P}_N,\mathbb{P})
\end{equation}
for every choice of empirical barycenter $\mu_{\mathbb{P}_N}$. Thus it remains to control $W_1(\mathbb{P}_N,\mathbb{P})$.

For any $r\in(0,1]$, choose an $r$-net $\{z_1,\dots,z_M\}$ of $K$ with $M=\mathcal N(K,W_2,r)$, and let $\pi:K\to\{z_1,\dots,z_M\}$ be a Borel nearest-point projection. Since both $\mathbb{P}$ and $\mathbb{P}_N$ are supported on $K$,
\[
W_1(\mathbb{P}_N,\mathbb{P})
\leq 2r+W_1(\pi_\#\mathbb{P}_N,\pi_\#\mathbb{P})
\leq 2r+\Delta\,\|\pi_\#\mathbb{P}_N-\pi_\#\mathbb{P}\|_{\rm TV}.
\]
For the multinomial empirical measure on $M$ points, the standard $l^1$ risk bound  (see, e.g., Devroye--Gy\"orfi~\cite[Chapter~3]{DevroyeGyorfi1985}) gives 
\[
\mathbb E\|\pi_\#\mathbb{P}_N-\pi_\#\mathbb{P}\|_{\rm TV}
\leq \frac12\sqrt{\frac{M}{N}}.
\]
Together with \eqref{eq:entropy_assumption}, this gives
\[
\mathbb E W_1(\mathbb{P}_N,\mathbb{P})
\leq 2r+\frac{\Delta}{2}\sqrt{\frac{C_{\rm ent}r^{-d_{\rm ent}}}{N}}.
\]
Optimizing in $r$ yields
\begin{equation}\label{eq:expected_empirical_w1}
\mathbb E W_1(\mathbb{P}_N,\mathbb{P})\leq C_E N^{-1/(d_{\rm ent}+2)},
\end{equation}
with $C_E$ depending only on $\Delta,C_{\rm ent}$ and $d_{\rm ent}$.

The map $(\rho_1,\dots,\rho_N)\mapsto W_1(\mathbb{P}_N,\mathbb{P})$ has bounded differences: changing one sample changes the empirical measure by at most $\Delta/N$ in $W_1$. McDiarmid's inequality~\cite{zbMATH01036755} and \eqref{eq:expected_empirical_w1} imply
\begin{equation}\label{eq:empirical_concentration}
\mathsf P\!\left(W_1(\mathbb{P}_N,\mathbb{P})
\leq C_E N^{-1/(d_{\rm ent}+2)}+\Delta\sqrt{\frac{\log(1/\delta)}{2N}}\right)
\geq1-\delta.
\end{equation}
If \eqref{eq:sample_complexity} holds with $C$ sufficiently large, the right-hand side in \eqref{eq:empirical_concentration} is at most $c_*\varepsilon^{q_\sigma}$. Combining this with \eqref{eq:empirical_stability} gives the claim.
\end{proof}

\subsection{A geometric entropy bound in the Alexandrov setting}

The polynomial covering-number condition in Proposition~\ref{prop:sample} is strong when $K$ is a subset of the full Wasserstein space. For $K=\mathcal{P}(\Omega)$ one should instead use the following exponential metric-entropy estimate, which comes directly from the geometry of the underlying Alexandrov space.

\begin{lemma}[Metric entropy of Wasserstein spaces over Alexandrov spaces]
\label{lem:alexandrov-entropy}
Let $(\Omega,\dist)$ be a compact $n$-dimensional Alexandrov space with curvature bounded below by $k$ and diameter bounded above by $D$. Then there exists a constant $C=C(n,k,D)>0$ such that, for every $0<\varepsilon<1/2$,
\[
    \log \mathcal N\bigl(\mathcal P(\Omega),W_2,\varepsilon\bigr)
    \le
    C\,\varepsilon^{-n}\log\frac{C}{\varepsilon}.
\]
Consequently, for every $\eta>0$, after increasing the constant we have
\[
    \log \mathcal N\bigl(\mathcal P(\Omega),W_2,\varepsilon\bigr)
    \le
    C_\eta\,\varepsilon^{-(n+\eta)}.
\]
This is an exponential entropy statement for the full Wasserstein space and is distinct from the polynomial covering-number assumption \eqref{eq:entropy_assumption}.
\end{lemma}

\begin{proof}
By the uniform covering theorem for compact Alexandrov spaces with fixed dimension, lower curvature bound and diameter bound (see, for instance, \cite[Chapter~10]{BuragoBuragoIvanov2001}), there exists $C=C(n,k,D)>0$ such that
\[
    \mathcal N(\Omega,\dist,r)\le C r^{-n},
    \qquad 0<r<1 .
\]
Let $\{x_1,\ldots,x_m\}$ be an $r$-net of $\Omega$ with $m\le C r^{-n}$, and choose a measurable nearest-point projection
\[
    \pi_r:\Omega\to \{x_1,\ldots,x_m\}.
\]
For every $\mu\in\mathcal P(\Omega)$, set $\mu^r := (\pi_r)_\#\mu$. Then $W_2(\mu,\mu^r)\le r$.

It remains to cover the set of probability measures supported on $\{x_1,\ldots,x_m\}$. We identify such measures with the probability simplex
\[
    \Delta_m
    :=
    \left\{
        p=(p_1,\ldots,p_m)\in [0,1]^m:
        \sum_{i=1}^m p_i=1
    \right\}.
\]
If $p,q\in \Delta_m$, then, since $\operatorname{diam}(\Omega)\le D$,
\[
    W_2^2(p,q)
    \le
    D^2 \|p-q\|_{\mathrm{TV}}
    =
    \frac{D^2}{2}\|p-q\|_1 .
\]
Therefore an $\ell^1$-net of $\Delta_m$ with mesh $\delta := \varepsilon^2/(2D^2)$ gives a $W_2$-net with radius at most $\varepsilon/2$.

The simplex $\Delta_m$ admits an $\ell^1$-net of mesh $\delta$ and cardinality at most $(C/\delta)^m$. Hence
\[
    \log \mathcal N(\Delta_m,W_2,\varepsilon/2)
    \le
    m \log\frac{C}{\delta}
    \le
    C r^{-n}\log\frac{C}{\varepsilon}.
\]
Taking $r=\varepsilon/2$ and using the triangle inequality gives
\[
    \log \mathcal N\bigl(\mathcal P(\Omega),W_2,\varepsilon\bigr)
    \le
    C\varepsilon^{-n}\log\frac{C}{\varepsilon}.
\]
Finally, since $\log(C/\varepsilon)\le C_\eta\varepsilon^{-\eta}$ for every $\eta>0$ and $0<\varepsilon<1/2$, the second bound follows.
\end{proof}

\begin{corollary}[Empirical barycenters over Alexandrov spaces]
\label{cor:empirical-alexandrov}
Assume that $(\Omega,\dist)$ is a compact $n$-dimensional Alexandrov space with curvature bounded below by $k$ and diameter bounded above by $D$. Let $\mathbb P\in \mathcal P(\mathcal P(\Omega))$, and let
\[
    \mathbb P_N := \frac1N\sum_{i=1}^N \delta_{\rho_i},
\]
where $\rho_1,\ldots,\rho_N$ are independent samples with law $\mathbb P$. Then there exists a constant $C=C(n,k,D)>0$ such that, for all $N\ge 3$,
\[
    \mathbb E W_1(\mathbb P_N,\mathbb P)
    \le
    C\left(\frac{\log\log N}{\log N}\right)^{1/n}.
\]
Consequently, if $\mathbb P$ satisfies Assumption~\ref{assump} and if $\mu_{\mathbb P_N}$ and $\mu_{\mathbb P}$ denote the corresponding Wasserstein barycenters, then for every $\sigma>0$ there exists $C_\sigma>0$ such that
\[
    \mathbb E W_2(\mu_{\mathbb P_N},\mu_{\mathbb P})
    \le
    C_\sigma
    \left(\frac{\log\log N}{\log N}\right)^{
        \frac{1}{n(12+\sigma)}
    } .
\]
\end{corollary}

\begin{proof}
Let $\mathsf Z := (\mathcal P(\Omega),W_2)$. By Lemma~\ref{lem:alexandrov-entropy},
\[
    \log \mathcal N(\mathsf Z,\varepsilon)
    \le
    C\varepsilon^{-n}\log\frac{C}{\varepsilon}.
\]
Let $\mathsf Z_\varepsilon$ be an $\varepsilon$-net of $\mathsf Z$, and let $\Pi_\varepsilon:\mathsf Z\to \mathsf Z_\varepsilon$ be a measurable nearest-point projection. Then
\[
    W_1(\mathbb P_N,\mathbb P)
    \le
    2\varepsilon
    +
    W_1\bigl((\Pi_\varepsilon)_\#\mathbb P_N,
             (\Pi_\varepsilon)_\#\mathbb P\bigr).
\]
The two measures in the second term are probability measures on the finite set $\mathsf Z_\varepsilon$. Since $\mathsf Z$ has diameter bounded in terms of $D$, the standard finite-state estimate gives
\[
    \mathbb E
    W_1\bigl((\Pi_\varepsilon)_\#\mathbb P_N,
             (\Pi_\varepsilon)_\#\mathbb P\bigr)
    \le
    C_D
    \sqrt{
        \frac{\mathcal N(\mathsf Z,\varepsilon)}{N}
    } .
\]
Choose
\[
    \varepsilon_N
    :=
    A\left(\frac{\log\log N}{\log N}\right)^{1/n},
\]
where $A=A(n,k,D)>0$ is sufficiently large. Then, for all $N\ge 3$,
\[
    \log \mathcal N(\mathsf Z,\varepsilon_N)
    \le
    \frac12 \log N .
\]
Therefore
\[
    \mathbb E W_1(\mathbb P_N,\mathbb P)
    \le
    2\varepsilon_N + C_D N^{-1/4}
    \le
    C\left(\frac{\log\log N}{\log N}\right)^{1/n}.
\]

The barycenter estimate follows from Theorem~\ref{barycenter}, which gives, for every $\sigma>0$,
\[
    W_2(\mu_{\mathbb P_N},\mu_{\mathbb P})
    \le
    C_\sigma
    W_1(\mathbb P_N,\mathbb P)^{1/(12+\sigma)}.
\]
Taking expectations and using Jensen's inequality, since $t\mapsto t^{1/(12+\sigma)}$ is concave, yields the claim.
\end{proof}

\begin{corollary}[Consistency of empirical barycenters]\label{cor:consistency}
Assume that $\mathbb{P}\in\mathcal{P}(\mathcal{P}(\Omega))$ satisfies Assumption~\ref{assump}. Let $(\rho_i)_{i\geq1}$ be i.i.d. random variables with law $\mathbb{P}$ on $(\Theta,\mathcal F,\mathsf P)$. For every $N$ and every $\omega\in\Theta$, let $\mu_{\mathbb{P}_N(\omega)}$ be any Wasserstein barycenter of $\mathbb{P}_N(\omega)$. Then
\[
W_2\bigl(\mu_{\mathbb{P}_N(\omega)},\mu_\mathbb{P}\bigr)\longrightarrow0
\qquad\text{for $\mathsf P$-a.e. }\omega.
\]
\end{corollary}

\begin{proof}
Since $\mathsf Z=(\mathcal{P}(\Omega),W_2)$ is compact and metric, Varadarajan's theorem~\cite{Varadarajan58} gives weak convergence $\mathbb{P}_N(\omega)\rightharpoonup\mathbb{P}$ almost surely. On a compact metric space, weak convergence is equivalent to convergence in $W_1$; hence
\[
W_1(\mathbb{P}_N(\omega),\mathbb{P})\to0
\qquad\text{for $\mathsf P$-a.e. }\omega.
\]
The conclusion follows by applying Theorem~\ref{barycenter} with $\mathbb{P}$ as the reference distribution and $\mathbb{P}_N(\omega)$ as the perturbation.
\end{proof}

\begin{appendices}
\section{Zero-order balance condition}\label{A}

The following global zero-order balance condition was first established in \cite{zbMATH05956494} and then generalized in \cite{zbMATH07819904}.

\begin{proposition}[Global zero-order balance]\label{balance}
There exists a $\mathbb{P}$-measurable map
\[
\mathcal{P}(\Omega)\ni\rho\mapsto \psi_\rho\in C(\Omega)
\]
such that, setting $\phi_\rho=\psi_\rho^c$, the following hold:
\begin{enumerate}[{\rm (i)}]
\item for $\mathbb{P}$-a.e.\ $\rho\in\mathcal{P}(\Omega)$, $(\phi_\rho,\psi_\rho)$ is a pair of Kantorovich potentials for $(\rho,\mu_{\mathbb{P}})$;
\item the map satisfies the global zero-order balance condition
\[
\int_{\mathcal{P}(\Omega)}\psi_\rho(y)\,\d\mathbb{P}(\rho)=0,\qquad \forall y\in\Omega.
\]
\end{enumerate}
\end{proposition}

\begin{proof}
We divide the proof into six steps.

\medskip
\noindent\textbf{Step 1: Definition and convexity of $H$.}
For $\varphi\in C(\Omega)$ define
\[
H(\varphi):=\inf\Bigl\{-\int_{\mathcal{P}(\Omega)}\E_\rho(\psi_\rho^c)\,\d\mathbb{P}(\rho):
(\psi_\rho)_\rho\in L^1(\mathbb{P};C(\Omega)),\ \int_{\mathcal{P}(\Omega)}\psi_\rho(\cdot)\,\d\mathbb{P}(\rho)=\varphi(\cdot)\Bigr\}.
\]
Since $\psi\mapsto -\psi^c$ is convex (as a supremum of affine functions), $H$ is convex. 
Fix $y_0\in\Omega$. For any admissible $(\psi_\rho)_\rho$, denote $D:=\diam(\Omega)$,
\[
-\psi_\rho^c(x)=\sup_{y\in\Omega}\{\psi_\rho(y)-c(x,y)\}\geq \psi_\rho(y_0)-\tfrac12D^2,
\]
whence $H(\varphi)\geq \varphi(y_0)-\tfrac12D^2>-\infty$. 
If $\|\varphi\|_{L^\infty}\leq 1$, taking $\psi_\rho\equiv\varphi$ yields $H(\varphi)\leq 1$. 
Thus $H$ is proper and bounded above on a neighbourhood of $0$; by standard convex analysis (cf.~\cite[Proposition~2.5]{zbMATH03504682}) it is therefore lower semicontinuous, and the Fenchel--Moreau theorem gives $H^{**}(0)=H(0)$.

\medskip
\noindent\textbf{Step 2: The barycenter as the convex conjugate.}
For $\mu\in\mathcal{P}(\Omega)$, Fubini--Tonelli and Kantorovich duality yield
\begin{equation}\label{eq:Hstar}
\begin{aligned}
H^*(\mu)
&=\sup_{\varphi\in C(\Omega)}\Bigl\{\E_\mu(\varphi)-H(\varphi)\Bigr\}\\
&=\sup_{(\psi_\rho)_\rho}\int_{\mathcal{P}(\Omega)}\bigl(\E_\mu(\psi_\rho)+\E_\rho(\psi_\rho^c)\bigr)\,\d\mathbb{P}(\rho)
\leq \int_{\mathcal{P}(\Omega)}W_2^2(\mu,\rho)\,\d\mathbb{P}(\rho).
\end{aligned}
\end{equation}
Conversely, by the Kuratowski--Ryll--Nardzewski measurable selection theorem there exists a $\mathbb{P}$-measurable map $\rho\mapsto\psi_\rho^\mu$ such that $((\psi_\rho^\mu)^c,\psi_\rho^\mu)$ is optimal for $(\rho,\mu)$. Centering $\tilde\psi_\rho^\mu:=\psi_\rho^\mu-\psi_\rho^\mu(y_0)$ preserves optimality and places $(\tilde\psi_\rho^\mu)_\rho\in L^1(\mathbb{P};C(\Omega))$. Substituting this family in \eqref{eq:Hstar} shows equality:
\[
H^*(\mu)=\int_{\mathcal{P}(\Omega)} W_2^2(\mu,\rho)\,\d\mathbb{P}(\rho).
\]
Moreover, by \cite[Lemma~2.1]{zbMATH05956494},
\[
H^{**}(0)=\sup_{\nu\in C(\Omega)^*}\{-H^*(\nu)\}
=\sup_{\mu\in\mathcal{P}(\Omega)}\{-H^*(\mu)\}
=-\inf_{\mu\in\mathcal{P}(\Omega)}H^*(\mu).
\]
Since $\mu_{\mathbb{P}}$ is the barycenter, we obtain
\begin{equation}\label{eq:Hzero}
-H(0)=H^{**}(0)=\inf_{\mu\in\mathcal{P}(\Omega)}\int_{\mathcal{P}(\Omega)} W_2^2(\mu,\rho)\,\d\mathbb{P}(\rho)
=\int_{\mathcal{P}(\Omega)} W_2^2(\mu_{\mathbb{P}},\rho)\,\d\mathbb{P}(\rho).
\end{equation}

\medskip
\noindent\textbf{Step 3: Normalization of a maximizing sequence.}
Let $\{(\psi_\rho^m)_\rho\}_{m\in\N}\subset L^1(\mathbb{P};C(\Omega))$ satisfy $\int_{\mathcal{P}(\Omega)}\psi_\rho^m(\cdot)\,\d\mathbb{P}(\rho)=0$ and
\[
\int_{\mathcal{P}(\Omega)}\E_\rho\bigl((\psi_\rho^m)^c\bigr)\,\d\mathbb{P}(\rho)\longrightarrow -H(0).
\]
Set $\hat\psi_\rho^m:=(\psi_\rho^m)^{\mathrm{cc}}$. Then $(\hat\psi_\rho^m)^c=(\psi_\rho^m)^c$ and $\hat\psi_\rho^m\geq\psi_\rho^m$. Define
\[
\alpha_m(y):=\int_{\mathcal{P}(\Omega)}\hat\psi_\rho^m(y)\,\d\mathbb{P}(\rho)\geq 0,
\qquad\bar\psi_\rho^m:=\hat\psi_\rho^m-\alpha_m.
\]
Then $\int\bar\psi_\rho^m\,\d\mathbb{P}=0$. Since the $c$-transform is order-reversing and $\bar\psi_\rho^m\leq\hat\psi_\rho^m$, we have $(\bar\psi_\rho^m)^c\geq(\hat\psi_\rho^m)^c=(\psi_\rho^m)^c$, hence the dual value does not decrease upon replacing $\psi^m$ by $\bar\psi^m$.

Next center at $y_0$: $\tilde\psi_\rho^m:=\bar\psi_\rho^m-\bar\psi_\rho^m(y_0)$. Then $\tilde\psi_\rho^m(y_0)=0$, $\int\tilde\psi_\rho^m\,\d\mathbb{P}=0$, and $(\tilde\psi_\rho^m)^c=(\bar\psi_\rho^m)^c+\bar\psi_\rho^m(y_0)$. Integrating and using $\int\bar\psi_\rho^m(y_0)\,\d\mathbb{P}=0$ shows the dual value is unchanged. Thus $\{(\tilde\psi_\rho^m)_\rho\}$ is still maximizing. Since $\hat\psi_\rho^m$ and $\alpha_m$ are $D$-Lipschitz, every $\tilde\psi_\rho^m$ is $2D$-Lipschitz and satisfies $\|\tilde\psi_\rho^m\|_{L^\infty}\leq 2D^2$.

\medskip
\noindent\textbf{Step 4: Compactness and extraction of a limit.}
Let
\[
K:=\{f\in C(\Omega):f(y_0)=0,\ \Lip(f)\leq 2D\}.
\]
By the Arzel\`a--Ascoli theorem, $K$ is compact in $(C(\Omega),\|\cdot\|_{L^\infty})$. Fix a dense set $\{y_j\}_{j\in\N}\subset\Omega$ containing $y_0$ and set $u_j^m(\rho):=\tilde\psi_\rho^m(y_j)$. Each sequence $\{u_j^m\}_{m}$ is bounded in $L^1(\mathbb{P})$. By the Koml\'os theorem and a standard diagonal argument, there exists a subsequence (still indexed by $m$) and functions $u_j\in L^1(\mathbb{P})$ such that the Ces\`aro means
\[
\check u_j^M:=\frac1M\sum_{m=1}^M u_j^m\longrightarrow u_j\quad\mathbb{P}\text{-a.e.},\qquad\forall j\in\N.
\]
Define $\check\psi_\rho^M:=\frac1M\sum_{m=1}^M\tilde\psi_\rho^m\in K$. For $\mathbb{P}$-a.e.\ $\rho$, the functions $\check\psi_\rho^M$ are equi-Lipschitz and converge pointwise on the dense set $\{y_j\}$; hence they converge uniformly to a unique $2D$-Lipschitz limit $\tilde\psi_\rho\in K$ with $\tilde\psi_\rho(y_0)=0$.

\medskip
\noindent\textbf{Step 5: The limit is admissible and optimal.}
The uniform convergence gives
\[
\lim_{M\to\infty}\|\check\psi_\rho^M-\tilde\psi_\rho\|_{L^\infty(\Omega)}=0
\quad\text{for }\mathbb{P}\text{-a.e. }\rho.
\]
Since each $\rho\mapsto\check\psi_\rho^M$ is $\mathbb{P}$-measurable and $C(\Omega)$ is separable, the limit map $\rho\mapsto\tilde\psi_\rho$ is also $\mathbb{P}$-measurable; boundedness by $2D^2$ yields $(\tilde\psi_\rho)_\rho\in L^1(\mathbb{P};C(\Omega))$. By dominated convergence,
\[
\int_{\mathcal{P}(\Omega)}\tilde\psi_\rho(y)\,\d\mathbb{P}(\rho)
=\lim_{M\to\infty}\int_{\mathcal{P}(\Omega)}\check\psi_\rho^M(y)\,\d\mathbb{P}(\rho)=0,\qquad\forall y\in\Omega.
\]
Moreover, $\|(\check\psi_\rho^M)^c-(\tilde\psi_\rho)^c\|_{L^\infty}\leq\|\check\psi_\rho^M-\tilde\psi_\rho\|_{L^\infty}\to0$, and by concavity of the $c$-transform together with Jensen's inequality,
\[
(\check\psi_\rho^M)^c=\Bigl(\frac1M\sum_{m=1}^M\tilde\psi_\rho^m\Bigr)^c
\geq\frac1M\sum_{m=1}^M(\tilde\psi_\rho^m)^c
\geq\frac1M\sum_{m=1}^M(\psi_\rho^m)^c.
\]
Integrating in $\rho$ and letting $M\to\infty$ yields
\[
\int_{\mathcal{P}(\Omega)}\E_\rho\bigl((\tilde\psi_\rho)^c\bigr)\,\d\mathbb{P}(\rho)\geq -H(0).
\]
Since $-H(0)$ is the supremum, equality holds; thus $(\tilde\psi_\rho)_\rho$ is a maximizer.

\medskip
\noindent\textbf{Step 6: Identification of Kantorovich potentials.}
By \eqref{eq:Hzero} and optimality of $\tilde\psi$,
\[
\int_{\mathcal{P}(\Omega)} W_2^2(\mu_{\mathbb{P}},\rho)\,\d\mathbb{P}(\rho)
=\int_{\mathcal{P}(\Omega)}\E_\rho\bigl((\tilde\psi_\rho)^c\bigr)\,\d\mathbb{P}(\rho).
\]
Kantorovich duality gives, for $\mathbb{P}$-a.e.\ $\rho$,
\[
W_2^2(\mu_{\mathbb{P}},\rho)\geq \E_\rho\bigl((\tilde\psi_\rho)^c\bigr)+\E_{\mu_{\mathbb{P}}}(\tilde\psi_\rho).
\]
Integrating in $\rho$ and using the balance condition $\int\tilde\psi_\rho(y)\,\d\mathbb{P}(\rho)=0$ (so that $\int\E_{\mu_{\mathbb{P}}}(\tilde\psi_\rho)\,\d\mathbb{P}=0$), we obtain
\[
\int_{\mathcal{P}(\Omega)} W_2^2(\mu_{\mathbb{P}},\rho)\,\d\mathbb{P}(\rho)
\geq\int_{\mathcal{P}(\Omega)}\E_\rho\bigl((\tilde\psi_\rho)^c\bigr)\,\d\mathbb{P}(\rho).
\]
As the two sides are equal, the pointwise inequality must be an equality for $\mathbb{P}$-a.e.\ $\rho$. Consequently $(\tilde\phi_\rho,\tilde\psi_\rho)$ with $\tilde\phi_\rho:=(\tilde\psi_\rho)^c$ is a pair of Kantorovich potentials for $(\rho,\mu_{\mathbb{P}})$, completing the proof.
\end{proof}
\end{appendices}

\addcontentsline{toc}{section}{References}

\providecommand{\href}[2]{#2}

\bigskip

\noindent\textbf{Declaration.} The authors declare that they have no conflict of interest and that the manuscript has no associated data. 
During the preparation of this manuscript, the authors used Kimi (Moonshot AI) and ChatGPT (OpenAI) for language polishing and stylistic refinement. 
All mathematical contents were thoroughly reviewed and verified by the authors.

\begin{thebibliography}{LPRS23}

\bibitem[AC11]{zbMATH05956494}
Martial Agueh and Guillaume Carlier, \emph{Barycenters in the {Wasserstein} space}, SIAM J. Math. Anal. \textbf{43} (2011), no.~2, 904--924, \href{https://doi.org/10.1137/100805741}{doi:10.1137/100805741}.

\bibitem[ACLP20]{AhidarLeGouicParis2020}
Adil Ahidar-Coutrix, Thibaut Le Gouic and Quentin Paris, \emph{Convergence rates for empirical barycenters in metric spaces: curvature, convexity and extendable geodesics}, Probab. Theory Related Fields \textbf{177} (2020), no.~1--2, 323--368, \href{https://doi.org/10.1007/s00440-019-00950-0}{doi:10.1007/s00440-019-00950-0}.

\bibitem[AB21]{AltschulerBoixAdsera2021}
Jason M. Altschuler and Enric Boix-Adser\`a, \emph{{Wasserstein} barycenters can be computed in polynomial time in fixed dimension}, J. Mach. Learn. Res. \textbf{22} (2021), no.~44, 1--19.



\bibitem[Ber08]{zbMATH05349267}
J\'er\^ome Bertrand, \emph{Existence and uniqueness of optimal maps on {Alexandrov} spaces}, Adv. Math. \textbf{219} (2008), no.~3, 838--851, \href{https://doi.org/10.1016/j.aim.2008.06.001}{doi:10.1016/j.aim.2008.06.001}.

\bibitem[BBI01]{BuragoBuragoIvanov2001}
Dmitri Burago, Yuri Burago and Sergei Ivanov, \emph{A course in metric geometry}, Graduate Studies in Mathematics, vol.~33, American Mathematical Society, Providence, RI, 2001.

\bibitem[CDM24]{zbMATH07819904}
Guillaume Carlier, Alex Delalande and Quentin M\'erigot, \emph{Quantitative stability of barycenters in the {Wasserstein} space}, Probab. Theory Related Fields \textbf{188} (2024), no.~3--4, 1257--1286, \href{https://doi.org/10.1007/s00440-023-01241-5}{doi:10.1007/s00440-023-01241-5}.

\bibitem[CD14]{CuturiDoucet2014}
Marco Cuturi and Arnaud Doucet, \emph{Fast computation of {Wasserstein} barycenters}, Proceedings of the 31st International Conference on Machine Learning, Proceedings of Machine Learning Research, vol.~32, PMLR, 2014, pp.~685--693.

\bibitem[DG85]{DevroyeGyorfi1985}
Luc Devroye and L\'aszl\'o Györfi, \emph{Nonparametric density estimation: the $L_1$ view}, Wiley Series in Probability and Mathematical Statistics, John Wiley \& Sons, New York, 1985.

\bibitem[ET76]{zbMATH03504682}
Ivar Ekeland and Roger Temam, \emph{Convex analysis and variational problems}, Studies in Mathematics and its Applications, vol.~1, North-Holland, Amsterdam, 1976.

\bibitem[GKO13]{GigliKuwadaOhta2013}
Nicola Gigli, Kazumasa Kuwada and Shin-ichi Ohta, \emph{Heat flow on {Alexandrov} spaces}, Comm. Pure Appl. Math. \textbf{66} (2013), no.~3, 307--331, \href{https://doi.org/10.1002/cpa.21431}{doi:10.1002/cpa.21431}.

\bibitem[HLZ25]{arXiv:2412.01190}
Bang-Xian Han, Deng-Yu Liu and Zhuo-Nan Zhu, \emph{On the geometry of {Wasserstein} barycenter {I}}, preprint, arXiv:2412.01190 [math.MG], 2025.

\bibitem[HZ26]{arXiv:2602.19175}
Bang-Xian Han and Zhuo-Nan Zhu, \emph{Stability of optimal transport on metric measure spaces}, preprint, arXiv:2602.19175 [math.MG], 2026.

\bibitem[Jia17]{zbMATH06797407}
Yin Jiang, \emph{Absolute continuity of {Wasserstein} barycenters over {Alexandrov} spaces}, Canad. J. Math. \textbf{69} (2017), no.~5, 1087--1108, \href{https://doi.org/10.4153/CJM-2016-035-8}{doi:10.4153/CJM-2016-035-8}.

\bibitem[KP17]{KimPass2017}
Young-Heon Kim and Brendan Pass, \emph{{Wasserstein} barycenters over {Riemannian} manifolds}, Adv. Math. \textbf{307} (2017), 640--683, \href{https://doi.org/10.1016/j.aim.2016.11.026}{doi:10.1016/j.aim.2016.11.026}.

\bibitem[LL17]{LeGouicLoubes2017}
Thibaut Le Gouic and Jean-Michel Loubes, \emph{Existence and consistency of {Wasserstein} barycenters}, Probab. Theory Related Fields \textbf{168} (2017), no.~3--4, 901--917, \href{https://doi.org/10.1007/s00440-016-0727-z}{doi:10.1007/s00440-016-0727-z}.

\bibitem[LPRS23]{LeGouicParisRigolletStromme2023}
Thibaut Le Gouic, Quentin Paris, Philippe Rigollet and Austin J. Stromme, \emph{Fast convergence of empirical barycenters in {Alexandrov} spaces and the {Wasserstein} space}, J. Eur. Math. Soc. (JEMS) \textbf{25} (2023), no.~6, 2229--2250, \href{https://doi.org/10.4171/JEMS/1234}{doi:10.4171/JEMS/1234}.

\bibitem[McD89]{zbMATH01036755}
Colin McDiarmid, \emph{On the method of bounded differences}, Surveys in Combinatorics, 1989 (Norwich, 1989), London Math. Soc. Lecture Note Ser., vol.~141, Cambridge Univ. Press, Cambridge, 1989, pp.~148--188.

\bibitem[Oht12]{Ohta2012}
Shin-ichi Ohta, \emph{Barycenters in {Alexandrov} spaces of curvature bounded below}, Adv. Geom. \textbf{12} (2012), no.~4, 571--587, \href{https://doi.org/10.1515/advgeom-2011-058}{doi:10.1515/advgeom-2011-058}.

\bibitem[OS94]{zbMATH00606998}
Yukio Otsu and Takashi Shioya, \emph{The {Riemannian} structure of {Alexandrov} spaces}, J. Differential Geom. \textbf{39} (1994), no.~3, 629--658.

\bibitem[PZ19]{PanaretosZemel2019}
Victor M. Panaretos and Yoav Zemel, \emph{Statistical aspects of {Wasserstein} distances}, Annu. Rev. Stat. Appl. \textbf{6} (2019), 405--431, \href{https://doi.org/10.1146/annurev-statistics-030718-104938}{doi:10.1146/annurev-statistics-030718-104938}.

\bibitem[Pet07]{zbMATH05342782}
Anton Petrunin, \emph{Semiconcave functions in {Alexandrov}'s geometry}, Surveys in Differential Geometry, Vol.~XI, Int. Press, Somerville, MA, 2007, pp.~137--201.

\bibitem[Pet11]{zbMATH06032507}
Anton Petrunin, \emph{{Alexandrov} meets {Lott}-{Villani}-{Sturm}}, M\"unster J. Math. \textbf{4} (2011), 53--64.

\bibitem[PC19]{PeyreCuturi2019}
Gabriel Peyr\'e and Marco Cuturi, \emph{Computational optimal transport: with applications to data science}, Found. Trends Mach. Learn. \textbf{11} (2019), no.~5--6, 355--607, \href{https://doi.org/10.1561/2200000073}{doi:10.1561/2200000073}.

\bibitem[RPDB12]{RabinPeyreDelonBernot2012}
Julien Rabin, Gabriel Peyr\'e, Julie Delon and Marc Bernot, \emph{{Wasserstein} barycenter and its application to texture mixing}, Scale Space and Variational Methods in Computer Vision, Lecture Notes in Computer Science, vol.~6667, Springer, Berlin, 2012, pp.~435--446, \href{https://doi.org/10.1007/978-3-642-24785-9_37}{doi:10.1007/978-3-642-24785-9\_37}.

\bibitem[San15]{Santambrogio2015}
Filippo Santambrogio, \emph{Optimal transport for applied mathematicians}, Progress in Nonlinear Differential Equations and their Applications, vol.~87, Birkh\"auser, Basel, 2015.

\bibitem[Var58]{Varadarajan58}
V. S. Varadarajan, \emph{On the convergence of sample probability distributions}, Sankhy\=a \textbf{19} (1958), no.~1--2, 23--26.


\end{thebibliography}
\end{document}